\documentclass[a4paper,12pt]{amsart}
\usepackage[english]{babel}
\usepackage{amsmath,amsthm,amssymb,amsfonts}
\usepackage{multicol}
 
\usepackage[pdftex]{color}
\usepackage[bookmarks=true,hyperindex,pdftex,colorlinks, citecolor=blue,linkcolor=blue, urlcolor=blue]{hyperref}
\usepackage[dvipsnames]{xcolor}
\usepackage{mathtools}
\usepackage{graphicx}
\usepackage[titletoc]{appendix}
\usepackage{tikz}	

\usepackage[normalem]{ulem}

\usetikzlibrary{matrix}

\newtheorem{theorem}{Theorem}[section]
\newtheorem{lemma}[theorem]{Lemma}

\newtheorem{proposition}[theorem]{Proposition}
\newtheorem{corollary}[theorem]{Corollary}

\theoremstyle{definition}
\newtheorem{definition}[theorem]{Definition}
\newtheorem{example}[theorem]{Example}

\theoremstyle{remark}
\newtheorem{remark}[theorem]{Remark}
\numberwithin{equation}{section}

\newcommand{\N}{\mathbb{N}}

\DeclareMathOperator{\diam}{diam\,}
\DeclareMathOperator{\co}{co}

\newcommand{\nn}[1]{{\left\vert\kern-0.25ex\left\vert\kern-0.25ex\left\vert #1 
		\right\vert\kern-0.25ex\right\vert\kern-0.25ex\right\vert}}

\renewcommand{\geq}{\geqslant}
\renewcommand{\leq}{\leqslant}

\newcommand{\norm}[1]{\left\Vert#1\right\Vert}

\newcommand{\Lip}{{\mathrm{Lip}}_0}
\newcommand{\justLip}{{\mathrm{Lip}}}

\newcommand{\dent}[1]{\operatorname{dent}\left(#1\right)}

\newcommand{\eps}{\varepsilon}

\begin{document}

\title{Daugavet points and $\Delta$-points in Lipschitz-free spaces}

\author[Jung]{Mingu Jung}
\address[Jung]{Department of Mathematics, POSTECH, Pohang 790-784, Republic of Korea \newline
\href{http://orcid.org/0000-0000-0000-0000}{ORCID: \texttt{0000-0003-2240-2855} }}
\email{\texttt{jmingoo@postech.ac.kr}}

\author[Rueda Zoca]{Abraham Rueda Zoca}
\address[Rueda Zoca]{Universidad de Murcia, Departamento de Matem\'aticas, Campus de Espinardo 30100 Murcia, (Spain)
 \newline
\href{https://orcid.org/0000-0003-0718-1353}{ORCID: \texttt{0000-0003-0718-1353} }}
\email{\texttt{abraham.rueda@um.es}}
\urladdr{\url{https://arzenglish.wordpress.com}}

\begin{abstract} 
We study Daugavet points and $\Delta$-points in Lipschitz-free Banach spaces. We prove that, if $M$ is a compact metric space, then $\mu\in S_{\mathcal F(M)}$ is a Daugavet point if, and only if, there is no denting point of $B_{\mathcal F(M)}$ at distance strictly smaller than two from $\mu$. Moreover, we prove that if $x$ and $y$ are connectable by rectifiable curves of length as close to $d(x,y)$ as we wish, then the molecule $m_{x,y}$ is a $\Delta$-point. Some conditions on $M$ which guarantee that the previous implication reverses are also obtained. As a consequence of our work, we show that Lipschitz-free spaces are natural examples of Banach spaces where we can guarantee the existence of $\Delta$-points which are not Daugavet points.
\end{abstract}

\thanks{The first author was supported by the Basic Science Research Program through the National Research Foundation of Korea (NRF) funded by the Ministry of Education (NRF-2019R1A2C1003857). The second author The research of Abraham Rueda Zoca was supported by Juan de la Cierva-Formaci\'on fellowship FJC2019-039973, by MTM2017-86182-P (Governmentof Spain, AEI/FEDER, EU), by MICINN (Spain) Grant PGC2018-093794-B-I00 (MCIU, AEI, FEDER, UE), by Junta de Andaluc\'ia Grant A-FQM-484-UGR18 and by Junta de Andaluc\'ia Grant FQM-0185.}

\subjclass[2010]{Primary 46B04; Secondary 46B20 }
\keywords{Lipschitz-free spaces; Daugavet points, $\Delta$-points; length spaces}

\maketitle

\thispagestyle{plain}

\section{Introduction}
 
Let $X$ be a Banach space and $T:X\longrightarrow X$ be a bounded operator. We say that $T$ satisfies the \textit{Daugavet equation} if
\begin{equation}\label{ecuadauga}
\Vert T+I\Vert=1+\Vert T\Vert,
\end{equation}
where $I$ denotes the identity operator. The study of Daugavet equation has been widely studied in the literature \cite{kkw,kssw,rtv} and references therein), even in a more general context, where $T$ is not necessarily a bounded operator (see e.g. \cite{cgmm,kmmw,sw}).

One of the most famous properties related to Daugavet equation, which precisely justifies its nomenclature, is the Daugavet property. Recall that a Banach space $X$ has the \textit{Daugavet property} if every rank-one operator satisfies the Daugavet equation. This property comes from the paper \cite{dau}, where it is proved that $C([0,1])$ enjoys this property. Since then, a lot of examples of Banach spaces enjoying the Daugavet property have appeared such as $\mathcal C(K)$ for a compact Hausdorff and perfect topological space $K$, $L_1(\mu)$ and $L_\infty(\mu)$ for a non-atomic measure $\mu$ or the space of Lipschitz functions $\Lip(M)$ over a metrically convex space $M$ (see \cite{ikw,kssw,werner} and the references therein for details). 

In \cite{ik2004} the following weaker property is considered: a Banach space $X$ is said to be a space with \textit{bad projections} if $\Vert I-P\Vert\geq 2$ holds for every rank-one projection $P:X\longrightarrow X$. This property was rediscovered later with the name of \textit{LD2P+} \cite{ahntt} and with the one of \textit{diametral local diameter two property} \cite{blrjca}.

One of the reasons why the above properties have attracted attention is because they have strong connections with the geometry of the unit ball of a Banach space. To be more precise, let us recall the following two results.

\begin{theorem}\label{theo:carad2pslices}\cite[Theorem 2.1]{kssw}
Let $X$ be a Banach space. The following assertions are equivalent:
\begin{enumerate}
\item $X$ has the Daugavet property.
\item For every $x\in S_X$, every $\varepsilon>0$ and every slice $S$ of $B_X$ there exists $y\in S$ such that
$$\Vert x-y\Vert>2-\varepsilon.$$
\end{enumerate}
\end{theorem}

\begin{theorem}\label{theo:carad2pslices}
\cite[Theorem 1.4]{ik2004} Let $X$ be a Banach space. The following assertions are equivalent:
\begin{enumerate}
\item $X$ is a space with bad projections.
\item For every $x\in S_X$, every $\varepsilon>0$ and every slice $S$ of $B_X$ containing $x$ there exists $y\in S$ such that
$$\Vert x-y\Vert>2-\varepsilon.$$
\end{enumerate}
\end{theorem}

The previous characterizations motivated the authors of \cite{ahlp} to introduce local versions of the above geometric properties, which will be central in the present paper. Given a Banach space $X$ and a point $x\in S_X$, it is said that the point $x$ is:
\begin{itemize}
\item a \textit{Daugavet point} if, given any slice $S$ of $B_X$ and any $\varepsilon>0$ then there exists $y\in S$ with $\Vert x-y\Vert>2-\varepsilon$.

\item a \textit{$\Delta$-point} if, given any slice $S$ of $B_X$ containing $x$ and any $\varepsilon>0$ then there exists $y\in S$ with $\Vert x-y\Vert>2-\varepsilon$.
\end{itemize}

In the previous language, a Banach space $X$ has the Daugavet property (respectively, is a Banach space with bad projections) if every point in $S_X$ is a Daugavet point (respectively, a $\Delta$-point). Deeper connections between Daugavet and $\Delta$-points with the Daugavet property are exhibited in \cite{ahlp}. Furthermore, examples of Daugavet and $\Delta$-points are exhibited in some classical Banach spaces in \cite[Section 3]{ahlp}.

The main aim of this paper is to study Daugavet points and $\Delta$-points in Lipschitz-free spaces (see formal definition in Section \ref{sect:notation}). In the last years, geometric properties around Daugavet property have been deeply studied (see. e.g. \cite{blr2018,gpr2018,lr2020,o2020,pr2018}).
Among these result, the one which is of particular interest to us is \cite[Theorem 3.5]{gpr2018} where it is proved that a Lipschitz-free space $\mathcal F(M)$ has the Daugavet property if, and only if, the metric space $M$ is \textit{length} (i.e. for every pair of distinct points $x,y\in M$ then $d(x,y)$ equals the infimum of the length of rectifiable curves in $M$ joining them). We will take advantage of a localization of this property in order to obtain sufficient conditions for a molecule $m_{x,y}$ to be a $\Delta$-point in Section \ref{sect:delta} which, in a large class of examples, will turn out to be an equivalence. Furthermore, in Section \ref{sect:daugavet} we obtain a characterization of when any element $\mu\in S_{\mathcal F(M)}$ is a Daugavet point in terms of a separation condition from the set of denting points. As a consequence of all the above mentioned results, we show that Lipschitz-free Banach spaces are a class where one can easily and naturally produce examples of $\Delta$-points which are not Daugavet points, in contrast with previously known examples which required the study of absolute sums of Banach spaces \cite[Corollary 5.5]{ahlp} or technical constructions of Banach spaces with 1-unconditional bases \cite[Theorem 3.1]{almt}.
 
Let us pass now to describe with more detail the content of the paper. In Section \ref{sect:notation} we present necessary notation together with some preliminary results. In Section \ref{sect:daugavet} we study Daugavet points in $\mathcal F(M)$ and in $\Lip(M)$. We prove in Theorem \ref{theo:charadaugapoint} that given a compact metric space $M$ then an element $\mu\in S_{\mathcal F(M)}$ is a Daugavet point if $\mu$ is at distance $2$ from every denting point of $B_{\mathcal F(M)}$. Furthermore, when $\mu$ is an element of the form $m_{x,y}$, we characterise the fact that $\mu$ is a Daugavet point in terms of a geometric condition on the points $x,y\in M$. We end the section with a brief discussion about Daugavet points in $\Lip(M)$ in connection with locality properties of Lipschitz functions. In Section \ref{sect:delta} we define the concept of \textit{connectable molecule} (see Definition \ref{defi:connect}), which can be seen as a localization of length property of metric spaces. We prove that if $m_{x,y}$ is a connectable molecule then it is a $\Delta$-point. This permits, on the one hand, to get a procedure to construct Lipschitz-free spaces with molecules which are $\Delta$- but not Daugavet points. On the other hand, in the search of conditions under which molecules which $\Delta$-points are connectable we prove that a molecule $m_{x,y}$ is a $\Delta$-point if, and only if for a given slice $S$ containing $m_{x,y}$ and $\eps >0$ there exist $u, v \in M$ with $0 < d(u,v) < \eps$ such that $m_{u,v} \in S$ (Theorem \ref{prop:deltapointlocal}). We end the section with some conditions under which every molecule which is a $\Delta$-point is connectable (Theorems \ref{theo:equideltacomconv} and \ref{theo:MLUR_equideltacomconv}) in subsets of strictly convex Banach spaces.

\section{Notation and preliminary results}\label{sect:notation}

We will consider only real Banach spaces. Given a Banach space $X$, we denote the closed ball (respectively, the sphere) of $X$ centered at $x \in X$ with radius $r >0$ by $B_X (x, r)$ (respectively, $S_X (x, r)$). If the center is $0$ and the radius is $1$, we simply denote the closed unit ball and the unit sphere by $B_X$ and $S_X$, respectively. 
We will also denote by $X^*$ the topological dual of $X$. By a \textit{slice of $B_X$} we mean a set of the following form
$$S(x^*,\alpha):=\{x\in B_X :x^*(x)>\sup x^*( B_X )-\alpha\}$$
where $x^*\in X^*$ and $\alpha>0$. If $X$ is a dual Banach space, the previous set will be called a \textit{$w^*$-slice} if $x^*$ belongs to the predual of $X$. Note that finite intersections of slices of $C$ (respectively of $w^*$-slices of $C$) form a basis for the inherited weak (respectively weak-star) topology of $C$.

Let $X$ be a Banach space and let $x\in S_X$. Recall that $x$ is a \textit{preserved extreme point} if $x$ is an extreme point of $B_{X^{**}}$. Also, $x$ is a \textit{denting point} if there exist slices of $B_X$ of arbitrarily small diameter containing $x$. We will denote by $\dent{B_X}$ the set of denting points of $B_X$.
 
It is known that $\Delta$-points in a Banach space $X$ can be characterized in the following way (see \cite[Lemma 2.1]{ahlp}). Let $x \in S_X$ be given. 
\begin{enumerate}
\item $x$ is a $\Delta$-point;
\item for every slice $S$ of $B_X$ with $x \in S$ and $\eps >0$, there exists $y \in S$ such that $\|x -y \| \geq 2-\eps$;
\item for every $x^* \in X^*$ with $x^* (x) = 1$, the projection $P = x^* \otimes x$ satisfies $\| I - P \| \geq 2$. 
\end{enumerate} 

The previous result allows us to obtain the following further characterization of $\Delta$-points, which is probably well known for specialist, but whose proof we include here for the sake of completeness. 

\begin{lemma}\label{lemma:deltnestedslices}
Let $X$ be a Banach space and $x \in S_X$ be a $\Delta$-point. For every $\eps >0$ and every slice $S = S(f, \alpha)$ of $B_X$ with $x \in S$ and $\frac{\alpha}{1-\alpha} < \eps$, there exists a slice $S(g, \alpha_1)$ of $B_X$ such that $S(g, \alpha_1) \subset S(f, \alpha)$ and $\|x - z \| \geq 2 -\eps$ for all $z \in S(g, \alpha_1)$. 
\end{lemma}

\begin{proof} The proof will follow the lines of \cite[Lemma 2.1]{kssw}. Choose $\eta >0$ so small that $\eta < 1- \frac{1-\alpha}{f(x)}$ and $\eta < \eps - \frac{\alpha}{1-\alpha}$. If $P := (f(x)^{-1} f ) \otimes x$, then $\| I - P \| \geq 2$ by \cite[Lemma 2.1]{ahlp}. It follows that there exists $y^* \in S_{X^*}$ such that $\|y^* - P^* y^* \| > 2 -\eta$. Define $g \in X^*$ as 
$g = \frac{ P^* y^* - y^* }{\|P^* y^* - y^* \|}$
and $\alpha_1 = 1- \frac{2-\eta}{\|P^* y^* - y^* \|}$. If $z \in S(g, \alpha_1)$, then we have that 
\[
y^* (x) \frac{f(z)}{f(x)}  - y^* (z)> 2 -\eta.
\]
With no loss of generality, we may assume that $y^* (x) > 0$ (noting that $y^*(x)$ cannot be zero). Then we get that $y^* (x) \frac{f(z)}{f(x)} > 1- \eta$, so $f(z) > (1-\eta) f(x) > 1 - \alpha$. 
Moreover, $\| f(x)^{-1} x - z \| > 2-\eta$ since $f(x)^{-1} y^* (x) - y^* (z) > 2-\eta$. This implies that 
\[
\| x - z \| > (2-\eta) - \left( \frac{\alpha}{1-\alpha} \right)  > 2 -\eps. 
\]
\end{proof} 

By using \cite[Lemma 2.1]{ik2004}, we can have an improvement of the previous result. 

\begin{lemma}\label{lem:delta_points_nested_slices}
Let $X$ be a Banach space. Then $x \in S_X$ is a $\Delta$-point if and only if for every $\eps >0$ and every slice $S = S(f, \alpha)$ of $B_X$ with $x \in S$, there exists a slice $S(g, \alpha_1)$ of $B_X$ such that $S(g, \alpha_1) \subset S(f, \alpha)$ and $\|x - z \| \geq 2 -\eps$ for all $z \in S(g, \alpha_1)$. 
\end{lemma} 

\begin{proof}
We only need to prove the ``only if" part. Let $\eps >0$ and a slice $S = S(f, \alpha)$ of $B_X$ with $x \in S$ be given. Pick any $\eta > 0$ satisfying that $\eta < \alpha$ and $\frac{\eta}{1-\eta} < \eps$. By \cite[Lemma 1.4]{ik2004}, there exists $h \in S_{X^*}$ such that $x \in S(h, \eta) \subset S(f, \alpha)$. Applying the above Lemma to the slice $S(h, \eta)$ and $\eps >0$, we may find a slice $S(g, \alpha_1)$ of $B_X$ such that $S(g,\alpha_1) \subset S(h, \eta)$ and $\| x - z \| \geq 2 -\eps$ for all $z \in S (g,\alpha_1)$. 
As $S(g,\alpha_1)$ is contained as well in $S( f, \alpha)$, we are done. 
\end{proof} 

\begin{remark}\label{remark:nestedslicedaugavet}
Similar estimates to the ones of the previous two lemmas allow to prove the following result: Let $X$ be a Banach space and $x\in S_X$ be a Daugavet point. Then, for every slice $S$ of $B_X$ and every $\varepsilon>0$, there exists a slice $T$ of $B_X$ contained in $S$ and such that
$$\Vert x-z\Vert>2-\varepsilon$$
holds for every $z\in T$.
\end{remark}

\begin{remark}\label{remark:nestedslice}
Let us explain our interest in the Lemma \ref{lem:delta_points_nested_slices}. Let $X$ be a Banach space and let $A$ be a subset of $B_X$ such that $\overline{\co}(A)=B_X$. Pick a $\Delta$-point $x\in S_X$. By definition, given a slice $S$ containing $x$, there are elements $y\in S$ such that $\Vert x-y\Vert>2-\varepsilon$. Lemma \ref{lem:delta_points_nested_slices} allows us to guarantee that one such element $y$ can be found in $A$. Indeed, Lemma \ref{lem:delta_points_nested_slices} implies the existence of a slice $T$ contained in $S$ satisfying that every element $y\in T$ satisfies $\Vert x-y\Vert>2-\varepsilon$. Now, since $\overline{\co}(A)=B_X$, $A$ intersects every slice of $B_X$, in particular $T\cap A\neq \emptyset$. This is a property of big relevance for Section \ref{sect:delta} (in particular, for Lemma \ref{lemma:lemanecdeltaclose}).
\end{remark}

Let us pass now to introduce necessary notation on Lipschitz-free spaces together with a preliminary result. Given a metric space $M$ and a point $x\in M$, we will denote by $B(x,r)$ (respectively, $S(x,r)$) the closed ball (respectively, sphere) centered at $x$ with radius $r$. Given two points $x,y\in M$ we define the \textit{metric segment} by
$$[x,y]:=\{z\in M: d(x,y)=d(x,z)+d(y,z)\}.$$
Let $M$ be a metric space with a distinguished point $0 \in M$.
The couple $(M,0)$ is commonly called a \emph{pointed metric space}.
By an abuse of language we will say only ``let $M$ be a pointed metric space'' and similar sentences.
The vector space of Lipschitz functions from $M$ to $\mathbb R$ will be denoted by $\justLip(M)$.
Given a Lipschitz function $f\in \justLip(M)$, we denote its Lipschitz constant by
\[ \norm{f}_{L} = \sup\left\{ \frac{|f(x)-f(y)|}{d(x,y)} : x,y\in M, x\neq y\right\}. \]
This is a seminorm on $\justLip(M)$ which is clearly a Banach space norm on the space $\Lip(M)\subset \justLip(M)$ of Lipschitz functions on $M$ vanishing at $0$.

We denote by $\delta$ the canonical isometric embedding of $M$ into $\Lip(M)^*$, which is given by $\langle f, \delta(x) \rangle =f(x)$ for $x \in M$ and $f \in \Lip(M)$. We denote by $\mathcal{F}(M)$ the norm-closed linear span of $\delta(M)$ in the dual space $\Lip(M)^*$, which is usually called the \textit{Lipschitz-free space over $M$}; for background on this, see the survey \cite{Godefroy_2015} and the book \cite{Weaver2} (where it receives the name of ``Arens-Eells space''). It is well known that $\mathcal{F}(M)$ is an isometric predual of the space $\Lip(M)$ \cite[p. 91]{Godefroy_2015}. We will write $\delta_x:=\delta(x)$ for $x\in M$, and use the name \textit{molecule} for those elements of $\mathcal F(M)$ of the form
$$m_{x,y}:=\frac{\delta_x-\delta_y}{d(x,y)}$$
for $x,y\in M$ with $x\neq y$. With a slight abuse of notation, we shall write $f (m_{x,y})$ to denote $\frac{f(x) - f(y)}{d(x,y)}$.

It is convenient to recall an important tool to construct Lipschitz functions: the classical McShane's extension theorem. It says that if $N \subseteq M$ and $f \colon N \longrightarrow \mathbb{R}$ is a Lipschitz function, then there is an extension to a Lipschitz function $F \colon M \longrightarrow \mathbb{R}$ with the same Lipschitz constant; see for example \cite[Theorem 1.33]{Weaver2}.

Let us now consider the following definitions.

\begin{definition}\label{defi:localetal}
Let $M$ be a metric space and $f\in \Lip(M)$.

\begin{enumerate}
\item  We say that $f$ is \textit{local} if, for every $\varepsilon>0$ there exists $u\neq v\in M$ with $0<d(u,v)<\varepsilon$ and such that $f(m_{u,v})> \|f\|-\varepsilon$.
\item We say that a point $t \in M$ is an \emph{$\eps$-point of $f$} if in every neighborhood $U \subset M$ of $t$, there exist $u \neq v \in U$ such that $f(m_{u,v}) > \|f\|-\varepsilon$. 
\item We say that $f$ is \emph{spreadingly local} if, for every $\eps >0$, there are infinitely many $\eps$-points of $f$. \end{enumerate}
\end{definition}

The above definitions come from the paper \cite{ikw}, where it was proved that if $M$ is a compact metric space then if every Lipschitz function is local then $\Lip(M)$ has the Daugavet property. Later, in \cite{gpr2018} it was proved that it is actually a characterization even when $M$ is complete.

Let us end with the following preliminary lemma, motivated by the ideas around the results of \cite[Section 3]{pr2018}, which tells us that one way of finding far molecules to a given element of $\mathcal F(M)$ is to look for close enough points. Namely, we have the following result.

\begin{theorem}\label{theo:prelipuntoscer}
Let $M$ be a metric space and let $u_n, v_n$ be two sequences in $M$ such that $0<d(u_n,v_n)$ for every $n$ and that $d(u_n,v_n)\rightarrow 0$. Then, for every $\mu\in S_{\mathcal F(M)}$ we get that 
$$\Vert \mu+m_{u_n,v_n}\Vert\rightarrow 2.$$
\end{theorem}

\begin{proof}
Assume by contradiction that there exists $\mu \in S_{\mathcal F(M)}$ and $\varepsilon_0>0$ such that
$$\Vert \mu+m_{u_n,v_n}\Vert\leq 2-\varepsilon_0$$
holds for every $n\in\mathbb N$. Since linear combinations of evaluation mappings are dense in $\mathcal F(M)$ we can assume with no loss of generality that $\mu=\sum_{i=1}^n \lambda_i \delta_{x_i}$ where $N:=\{x_1,\ldots, x_n\}\subseteq M\setminus \{0\}$. Pick $f\in S_{\Lip(M)}$ with $f(\mu)=1$. 

Define $\theta:=\inf\limits_{x\neq y\in N} d(x,y)>0$. Up to taking a subsequence, we can assume that $d(u_n,v_n)\leq \frac{\theta}{2n}$ holds for every $n\in\mathbb N$. Hence
$$d(x,y)+d(u_n,v_n)\leq d(x,u_n)+d(y,v_n)+2d(u_n,v_n)\leq d(x,u_n)+d(y,v_n)+\frac{1}{n}(d(x,y)+d(u_n,v_n)),$$
so 
\[
\left(1-\frac{1}{n} \right)(d(x,y)+d(u_n,v_n))\leq d(x,u_n)+d(y,v_n)
\]
holds for every $x,y\in N$ with $x \neq y$ and every $n \in N$. 

We define, for every $n\in\mathbb N$, a Lipschitz function $g_n$ on $N \cup \{ u_n, v_n \}$ as $g_n (x) = f(x)$ for every $x \in N$, 
\[
g_n (u_n):= \inf\limits_{x\in N} \left(g_n(x)+\frac{1}{1-\frac{1}{n}}d(x, u_n)\right)
\]
and
\[
g_n(v_n):=\sup\limits_{x\in N\cup\{u_n\}} \left(g_n(x)-\frac{1}{1-\frac{1}{n}}d(x,v_n)\right).
\]
Notice that $\Vert g_n\Vert\leq \frac{1}{1-\frac{1}{n}}$ for every $n\in\mathbb N$ (see, for example, \cite[Proposition 1.32]{Weaver2}). Fix $n \in \N$ so large that $1 - \frac{1}{n} > 1 - \frac{\eps_0}{2}$. 
By McShane's theorem, we can extend to the whole $M$ without increasing its Lipschitz norm. Since $g_n$ agrees with $f$ on $N$, we get that $g_n(\mu)=1$. We claim that $g_n(u_n)-g_n(v_n)\geq d(u_n,v_n)$ holds for every $n\in\mathbb N$ (equivalently $g_n(m_{u_n,v_n})\geq 1$). By definition, there are $z\in N$ and $z'\in N\cup\{u_n\}$ such that $g_n(u_n)=f(z)+d(z,u_n)$ and $g_n(v_n)=g_n(z')-d(z',v_n)$. If $z'=u_n$, then $g_n(u_n)-g_n(v_n)=d(u_n,v_n)$. If $z'\in N$, we have that 
\[
\begin{split}
g_n(u_n)-g_n(v_n)& =f(z)-f(z')+\frac{1}{1-\frac{1}{n}}(d(z,u_n)+d(z',v_n))\\
&\geq f(z)-f(z')+d(z,z')+d(u_n,v_n) \geq d(u_n,v_n)
\end{split}
\]
since $f(z)-f(z') + d(z,z') \geq 0$. 
Now
$$2-\varepsilon_0\geq \Vert  \mu +m_{u_n ,v_n}\Vert\geq \frac{g_n(\mu + m_{u_n,v_n})}{\Vert g_n\Vert} \geq 2 \left(1-\frac{1}{n}\right),$$
a contradiction.
\end{proof}

\section{Daugavet points}\label{sect:daugavet}

In this section we will focus on studying Daugavet points in $\mathcal F(M)$ as well as in $\Lip(M)$. Let us start with the following easy observation, which says that Daugavet points have to be far from the set of denting points. 

\begin{proposition}\label{prop:Daugavet_point_distance_denting_point}
Let $X$ be a Banach space and $x_0 \in S_X$ be a Daugavet point. Then, for every $y\in \dent{B_X}$, we have that $d(x,y)=2$. 
\end{proposition} 

\begin{proof}
Suppose that there exists $y \in \dent{B_X}$ and $\eps >0$ such that $d ( x_0, y ) \leq 2 - \eps$. Choose a slice $S$ containing $y$ so that $\diam (S) < \frac{\eps}{2}$. Note that  
\[
d( x_0, z) \leq d(x_0, y) + d(y, z) < 2-\eps + \frac{\eps}{2} = 2 - \frac{\eps}{2}.
\]
for every $z \in S$. This implies that $x_0$ cannot be a Daugavet point. 
\end{proof} 

In general, the converse of the previous proposition is false. Indeed, $\dent{B_{\ell_\infty}}=\emptyset$, so $d(x,y)=2$ for every $x\in B_{\ell_\infty}$ and every $y\in \dent{B_{\ell_\infty}}$. However, $\ell_\infty$ fails the Daugavet property (see e.g. \cite[P. 78]{werner}), so there are elements in $B_{\ell_\infty}$ which are not Daugavet points.

In spite of the previous example, we will prove that the previous behavior does not occur for the class of Lipschitz-free spaces over compact metric spaces.

\begin{theorem}\label{theo:charadaugapoint}
Let $M$ be a compact metric space and $\mu\in S_{\mathcal F(M)}$. The following assertions are equivalent:
\begin{enumerate}
\item $\mu$ is a Daugavet point.
\item For every $\nu\in \dent{B_{\mathcal F(M)}}$ then $d(\mu,\nu)=2$.\end{enumerate}

Moreover, if $\mu$ is of the form $m_{x,y}$ for certain $x\neq y\in M$, then the previous two are equivalent to:

\begin{itemize}

\item [(3)] If $u,v\in M$ satisfy that $[u,v]=\{u,v\}$ then
$$d(x,y)+d(u,v)\leq \min\{d(x,u)+d(y,v), d(x,v)+d(y,u)\}.$$
\end{itemize}
\end{theorem}

\begin{proof}
It is clear from Proposition \ref{prop:Daugavet_point_distance_denting_point} that (1)$\Rightarrow$(2). 

(2)$\Rightarrow$(1). Pick a slice $S=S(f,\alpha)$, where $f\in S_{\Lip(M)}$. Then we have two possibilities for $f$.

\begin{itemize}
\item $f$ is not local. Then by \cite[Lemma 3.13]{cgmr2020} $f$ attains its norm at a molecule $m_{u,v}$ which is strongly exposed point (in particular, it is a denting point). By the assumptions $\Vert \mu-m_{u,v}\Vert=2$ and, since the norm attaining condition, $m_{u,v}\in S$, and we are done in this case.

\item $f$ is local. In such a case, by definition, we can find a pair of sequences $u_n\neq v_n$ with $0<d(u_n,v_n)\rightarrow 0$ and such that $m_{u_n,v_n}\in S$ holds for every $n\in\mathbb N$. By Theorem \ref{theo:prelipuntoscer} we get that $\Vert \mu +  m_{u_n,v_n}\Vert\rightarrow 2$, and we are done.
\end{itemize}

(2)$\Leftrightarrow$(3). Assume now that $\mu=m_{x,y}$. Note that given $u,v\in M$, then $u,v$ satisfy that $[u,v]=\{u,v\}$ if, and only if, $m_{u,v}$ is an extreme point of $B_{\mathcal F(M)}$ \cite[Theorem 3.2]{appp2020}, which is in turn equivalent to be a preserved extreme point since $M$ is compact \cite[Theorem 4.2]{ag2019}, which in turn is equivalent to being a denting point by \cite[Theorem 2.4]{gppr2018}. Moreover, by \cite[Theorem 2.4]{ar2020} we get that $\Vert m_{x,y}\pm m_{u,v}\Vert=2$ is equivalent to the inequality $d(x,y)+d(u,v)\leq \min\{d(x,u)+d(y,v), d(x,v)+d(y,u)\}$.

From all these facts, (2) and (3) are equivalent.\end{proof}

Next, we exhibit an example of a metric space $M$ such that $\mathcal  F(M)$ does not have the Daugavet property but, on the other hand, there exists a Daugavet point $m_{x, y}$ in it.

\begin{example}\label{examp:daupoinnoda}
Let $M:=\{-1\}\cup [0,1]\subseteq \mathbb R$ and let $y=0, x=1$. Then $m_{x,y}$ is a Daugavet point. Indeed, if $u,v$ are such that $\{z\in M: d(z,u)+d(z,v)=d(u,v)\}=\{u,v\}$ then, up a relabeling, $u=0$ and $v=-1$. Moreover, notice that
$$d(x,y)+d(u,v)=2\leq \min\{d(x,u)+d(y,v), d(x,v)+d(y,u)\}=\min\{2,2\}=2.$$
By Theorem \ref{theo:charadaugapoint}, $m_{x,y}$ is a Daugavet point. However, it is easy to see that $\mathcal F(M)$ does not have the Daugavet property because it is clearly not a length space \cite[Theorem 3.5]{gpr2018}.
\end{example}

Now we turn to a brief discussion on Duagavet points and its $w^*$-version (see Definition \ref{defi:w*daugavetpoint}) on $\Lip(M)$. First of all, a local argument in \cite[Theorem 3.1]{ikw} yields, following word-by-word the proof, the following result.

\begin{proposition}\label{prop:daugaLipikw}
Let $M$ be a complete metric space and let $f \in S_{\Lip (M)}$. If $f$ is spreadingly local, then $f$ is a Daugavet point. 
\end{proposition} 

It is natural to wonder whether the previous proposition holds if $f$ is merely a local Lipschitz function. We do not know the answer. Note that, the main difficulty for studying Daugavet points in spaces of Lipschitz functions is that it is not known a good description of weak topology in $\Lip(M)$ (which makes difficult to determine whether or not a Lipschitz function belongs to a given (weak) slice). Because of this reason, we move to study a weak-star-version of the concept of Daugavet point in the following sense.

\begin{definition}\label{defi:w*daugavetpoint}
Let $X$ be a Banach space. An element $x^*\in S_{X^*}$ is said to be a \textit{$w^*$-Daugavet point} if, given a $w^*$-slice $S$ of $B_{X^*}$ and any $\varepsilon>0$, there exists $y^*\in S$ such that $\Vert x^*-y^*\Vert>2-\varepsilon$.
\end{definition}

Apart from being a natural generalisation of the concept of Daugavet point, the previous definition has deep connections with Banach spaces enjoying the Daugavet property thanks to the celebrated work \cite{kssw}. Indeed, \cite[Lemma 2.2]{kssw} says that a Banach space $X$ has the Daugavet property if, and only if, every element of $S_{X^*}$ is a $w^*$-Daugavet point.

The following theorem confirms, for the case of $w^*$ Daugavet point, our initial intuition about local Lipschitz functions. The proof will use ideas coming from \cite[Theorem 2.4]{blr2018}.

\begin{theorem}
Let $M$ be a complete metric space. If $f \in S_{\Lip (M)}$ is local, then $f$ is a $w^*$-Daugavet point. 
\end{theorem}

\begin{proof} 
Let $u_n$ and $v_n$ be distinct points in $M$ such that $d(u_n, v_n) < \frac{1}{n}$ and $f(m_{u_n, v_n}) > 1 - \frac{1}{n}$ for each $n \in \N$. 
Let us denote $A:=  \{ u_n : n \in \N \}$, and let $S = S (\mu, \alpha)$ be a $w^*$-slice of $B_{\Lip (M)}$ with $\mu \in \mathcal{F} (M)$ and $\alpha >0$. Without loss of generality, we may assume that $\mu = \sum_{i=1}^N \beta_i \delta _{x_i}$ with $\beta_i \in \mathbb{R}$ and $x_i \in M \setminus \{0\}$ for each $1 \leq i \leq N$. Fix an element $g \in S$. Our aim is to show that there exists $\tilde{h} \in S$ so that $\| \tilde{h} - f \|$ is sufficiently close to $2$.
We divide the proof into two cases, depending on the set $A$. 

\noindent \textbf{CASE I.} Assume that the set $A'$ of cluster points of $A$ is non-empty, i.e., there exists $p \in M$ which is a limit point of $A$. 
Choose a sequence $(u_{n_j} )_{j \in \N}$ in $A$ so that $u_{n_j}$ converges to $p$. Up to considering the isometric mapping $f\to f-f(p)$, we may assume that $p = 0$. 
Let $\eps \in (0,1)$ be given so that $(1+\eps)^{-1} {g(\mu)} > 1-\alpha.$
Pick $r \in (0, \eps)$ so small that $0 \notin B ( x_i , r)$ for every $1 \leq i \leq N$ and $B ( x_i , r) \cap B (x_j , r) =\emptyset$ for each $ 1 \leq i \neq j \leq N$.
Choose $\eta \in (0, r)$ sufficiently small so that 
\[
 \frac{\eta}{r-\eta} < {\eps} \quad \text{and} \quad \frac{2L}{\left(r-\frac{\eta}{2}\right)^2} \eta < \frac{\eps}{2},  
\]
where $L = \sup \{ |g(x)| :  x \in \{x_i, 1\leq i \leq n\} \cup B (0, \eps) \} > 0$. 
Let $J \in \N$ so large so that 
\[
\frac{1}{n_J} < \frac{\eta}{2} \quad  \text{and} \quad d(u_{n_J}, 0) < \frac{\eta}{2}. 
\]
This choice of $J \in \N$ implies, in particular, that $u_{n_J}$ and $v_{n_J}$ are not one of $x_i$'s. 
Define a function $h$ on $\{x_i : 1 \leq i \leq N\} \cup \{u_{n_J},v_{n_J}\} \subset M$ as $h(x_i) = g(x_i)$ for each $1 \leq i \leq N$, $h(u_{n_J}) = g(u_{n_J})$ and $h(v_{n_J}) = g(u_{n_J}) - d(u_{n_J},v_{n_J})$. Observe that $|h(m_{x_i, x_j})| = |g (m_{x_i, x_j})| \leq \|g \| \leq 1$ for $1 \leq i \neq j \leq N$. Also, note that $|h(m_{u_{n_J}, x_i})| = |g(m_{u_{n_J}, x_i})| \leq \|g \| \leq 1$ for $1 \leq i \leq N$. Moreover, 
\begin{align*} 
\left| \frac{ h(v_{n_J}) - h(x_i) }{ d(v_{n_J}, x_i) } \right| &= \left| \frac{g(u_{n_J})-g(x_i)}{d(v_{n_J},x_i)} - \frac{d(u_{n_J},v_{n_J})}{d(v_{n_J},x_i)} \right| \\
 &\leq \frac{ |g(u_{n_J})-g(x_i)|}{d(u_{n_J},x_i) - {n_J}^{-1}} + \frac{{n_J}^{-1}}{d(v_{n_J}, x_i)} \\
&< \frac{ |g(u_{n_J})-g(x_i)|}{d(u_{n_J},x_i) - \frac{\eta}{2}} + \frac{\eta}{2(r-\eta)}
\end{align*} 
as 
\[
d(v_{n_J}, x_i) \geq d(x_i, 0) - d(v_{n_J}, 0) \geq r - (d (u_{n_J}, 0) + d(u_{n_J}, v_{n_J}) )  > r - \eta. 
\]
Applying Tayor's theorem to $t \in [0,\frac{\eta}{2}] \mapsto \frac{|g(u)-g(x_i)|}{d(u,x_i) - t}$, we have that 
\begin{align*}
\frac{ |g(u_{n_J})-g(x_i)|}{d(u_{n_J},x_i) - \eta} &\leq \frac{ |g(u_{n_J})-g(x_i)|}{d(u_{n_J},x_i)} + \frac{ |g(u_{n_J})-g(x_i)|}{(d(u_{n_J},x_i) - \xi)^2} \eta \quad (\text{for some } \xi \in \left(0, \frac{\eta}{2} \right) ) \\
&\leq \|g \| + \frac{2L}{\left(r-\frac{\eta}{2}\right)^2} \eta < 1 + \frac{\eps}{2}. 
\end{align*} 
It follows that 
\[
\left| \frac{ h(v_{n_J}) - h(x_i) }{ d(v_{n_J}, x_i) } \right| < 1 + \frac{\eps}{2} + \frac{\eps}{2} = 1 +\eps.
\]
So, $h$ is Lipschitz on its domain with Lipschitz constant less than or equal to $1+\eps$. Extend $h$ to a Lipschitz function on $M$ with the same norm, still denoted by $h$. Notice that $\tilde{h}:= (1+\eps)^{-1} h \in B_{\Lip (M)}$ and $\tilde{h} (\mu) = (1+\eps)^{-1} g(\mu) > 1-\alpha$; hence $\tilde{h} \in S$. On the other hand, 
\[
\| f - \tilde{h} \| \geq f(m_{u_{n_J},v_{n_J}}) - \tilde{h}(m_{u_{n_J},v_{n_J}}) >\left( 1 - \frac{1}{{n_J}} \right) + \frac{1}{1+\eps} > 1 - \frac{\eps}{2} + \frac{1}{1+\eps}. 
\]
As $\eps > 0$ can be arbitrarily small, we have that $f$ is a $w^*$-Daugavet point. 

\noindent \textbf{CASE II.} Assume that there is no limit point of $A$ in $M$. Let $\eps \in (0,1)$ be given so that $(1+\eps)^{-1} {g(\mu)} > 1-\alpha.$
Choose $r \in (0,\eps)$ such that $B ( x_i , r) \cap B (x_j , r) =\emptyset$ for each $ 1 \leq i \neq j \leq N$ and $\left(B(x_i,  r) \setminus \{x_i\} \right) \cap A =\emptyset$ for every $1 \leq i \leq N$. 
Choose $\eta \in (0, r)$ so small so that 
\[
\max\left\{ \frac{\eta}{r-\eta}, \frac{2\eta}{r} \right\} < {\eps} \quad \text{and} \quad \frac{2L}{\left(r-\frac{\eta}{2}\right)^2} \eta < \frac{\eps}{2},  
\]
where $L = \sup \{ |g(x)| :  x \in \{x_i, 1\leq i \leq n\} \cup B (0, \eps) \} > 0$.
Let $J \in \N$ so large so that 
\[
\frac{1}{n_J} < \frac{\eta}{2} \quad  \text{and} \quad d(u_{n_J}, 0) < \frac{\eta}{2}. 
\]
Suppose that $u_{n_J} \neq x_i$ for every $1 \leq i \leq N$. Then $d(u_{n_J}, x_i ) \geq r$ and 
\[
d(v_{n_J}, x_i) \geq r - d(u_{n_J}, v_{n_J}) > r - \frac{\eta}{2} > \frac{r}{2}
\] 
for each $1 \leq i \leq N$. 
Define a function $h_1$ on $\{x_i : 1 \leq i \leq n\} \cup \{u_{n_J},v_{n_J}\} \subset M$ as $h_1(x_i) = g(x_i)$ for each $1 \leq i \leq n$, $h_1(u_{n_J}) = g(u_{n_J})$ and $h_1(v_{n_J}) = g(u_{n_J}) - d(u_{n_J},v_{n_J})$. Observe that $|h_1(m_{x_i, x_j})| = |g (m_{x_i, x_j})| \leq \|g \| \leq 1$ for $1 \leq i \neq j \leq N$. Also, note that $|h_1(m_{u_{n_J}, x_i})| = |g(m_{u_{n_J}, x_i})| \leq \|g \| \leq 1$ and that 
\begin{align*} 
\left| \frac{ h_1(v_{n_J}) - h_1 (x_i) }{ d(v_{n_J}, x_i) } \right| &= \left| \frac{g(u_{n_J})-g(x_i)}{d(v_{n_J},x_i)} - \frac{d(u_{n_J},v_{n_J})}{d(v_{n_J},x_i)} \right| \\
 &\leq \frac{ |g(u_{n_J})-g(x_i)|}{d(u_{n_J},x_i) - {n_J}^{-1}} + \frac{{n_J}^{-1}}{d(v_{n_J}, x_i)} \\
&< \frac{ |g(u_{n_J})-g(x_i)|}{d(u_{n_J},x_i) - \frac{\eta}{2}} + \frac{\eta}{r}
\end{align*} 
for $1 \leq i \leq N$. 
On the other hand, if $u_{n_J} = x_i$ for some $1 \leq i \leq N$, without loss of generality, we may assume that $u_{n_J} = x_1$. Note that $v_{n_J} \in B (x_1, \frac{\eta}{2})$; hence $d(v_{n_J}, x_i) \geq r$ for every $2 \leq i \leq N$. Define $h_2$ on $\{x_i : 1 \leq i \leq N \} \cup \{ v_{n_J} \}$ as $h_2 (x_i) = g(x_i)$ for each $1 \leq i \leq N$ and $h_2 (v_{n_J}) = g(x_1) - d(x_1,v_{n_J})$. Note that 
\begin{align*}
\left| \frac{ h_2 (v_{n_J}) - h_2 (x_i) }{ d(v_{n_J}, x_i) } \right| &= \left| \frac{g(x_1)-g(x_i)}{d(v_{n_J},x_i)} - \frac{d(x_1,v_{n_J})}{d(v_{n_J},x_i)} \right| \\
&< \frac{ |g(x_1)-g(x_i)|}{d(x_1,x_i) - \frac{\eta}{2}} + \frac{\eta}{2r}
\end{align*} 
for every $2 \leq i \leq N$. 
In both cases, arguing as before, we can obtain a Lipschitz function $\tilde{h} \in S$ so that $\| f - \tilde{h} \| \geq 2 - C \eps$ for some constant $C$. 
\end{proof}

\section{$\Delta$-points}\label{sect:delta}

In this section we move to study $\Delta$-points in Lipschitz-free spaces. For this, we will take our motivation from the results of \cite{gpr2018}, where a metric characterization of when $\mathcal F(M)$ has the Daugavet property is given in terms of a metric property depending purely on $M$, the property of $M$ being length. Taking a look to the definition of length spaces, we consider the following local concept of length space.

\begin{definition}\label{defi:connect}
Let $M$ be a metric space and $x\neq y\in M$. We say that the points $x$ and $y$ are \textit{connectable} if given $\varepsilon>0$ there exists a $1$-Lipschitz mapping $\alpha:[0,d(x,y)+\varepsilon]\longrightarrow M$ with $\alpha(0)=y$ and $\alpha(d(x,y)+\varepsilon)=x$.
\end{definition}

Notice that a metric space $M$ is length if and only if every pair of distinct points is connectable. This property is equivalent to the fact that for every $x, y \in M$ and for every $\delta >0$ the set 
\[
\text{Mid} (x, y, \delta) := B\left(x, \frac{1+\delta}{2} d(x,y)\right) \cap B \left(y, \frac{1+\delta}{2} d(x,y) \right) 
\]
is non-empty (see \cite[Lemma 3.2]{gpr2018}).

Our interest in this definition is the following result.

\begin{proposition}\label{prop:connectdelta}
Let $M$ be a metric space and let $x\neq y$ be a pair of points in $M$ which are connectable. Then $m_{x,y}$ is a $\Delta$-point.
\end{proposition}

\begin{proof}
Assume with no loss of generality that $d(x,y)=1$. Pick a slice $S=S(f,\alpha)$ containing $m_{x,y}$. Let us find $u\neq v$ such that $m_{u,v}\in S$ and $\Vert m_{x,y}-m_{u,v}\Vert\approx 2$. To this end, find $0<\beta<\alpha$ such that $f(m_{x,y})>1-\beta$ and take $\eta>0$ with $\frac{1-\beta}{1+\eta}>1-\alpha$ and take $\alpha:[0,1+\eta]\longrightarrow M$ to be a $1$-Lipschitz curve such that $\alpha(0)=y$ and $\alpha(1+\eta)=x$. Hence $f\circ \alpha:[0,1+\varepsilon]\longrightarrow \mathbb R$ is a $1$-Lipschitz map, so it is differentiable almost everywhere. Moreover
$$1-\beta<f(\alpha(1+\eta))-f(0)=\int_0^{1+\eta} (f\circ \alpha)'\leq (1+\eta)\Vert (f\circ \alpha)'\Vert_\infty,$$
so there exists $t_0\in [0,1+\eta]$ such that $(f\circ \alpha)'(t_0)>\frac{1-\beta}{1+\eta}>1-\alpha$.
Now pick $\varepsilon>0$. By the definition of derivative and the previous condition we can find $t\in [0,1+\varepsilon]$ with $0<d(t,t_0)<\varepsilon$ and such that $\frac{f(\alpha(t))-f(\alpha(t_0))}{t-t_0}>1-\alpha$. Now
\[
\begin{split}
1-\alpha<\frac{f(\alpha(t))-f(\alpha(t_0))}{t-t_0}& =\frac{f(\alpha(t))-f(\alpha(t_0))}{d(\alpha(t),\alpha(t_0))}\frac{d(\alpha(t),\alpha(t_0))}{t-t_0}\\
& \leq f(m_{\alpha(t),\alpha(t_0)})\Vert \alpha\Vert_L\leq f(m_{\alpha(t),\alpha(t_0)}),
\end{split}
\]
which implies that $m_{\alpha(t),\alpha(t_0)}\in S$. Moreover $d(\alpha(t),\alpha(t_0))<\varepsilon$. Summarising we have proved that we can find molecules in $S$ where the defining points are arbitrarily close. By Theorem \ref{theo:prelipuntoscer} we get that $m_{x,y}$ is a $\Delta$-point.
\end{proof}

Next we aim to get conditions under which connectibility is equivalent to the fact that the corresponding molecule is a $\Delta$-point. However let us obtain first that, from the previous proposition and from Theorem \ref{theo:charadaugapoint}, there are $\Delta$-points that are not Daugavet points in the context of Lipschitz-free spaces. This establishes a big difference between the class of Lipschitz-free spaces and $L_1$-spaces, where every $\Delta-$point is a Daugavet point \cite[Theorem 3.1]{ahlp}.

\begin{example}\label{examp:deltanotdauga}
Let $0<r<1$ and define $M:=[0,1]\times \{0\}\cup\{(0,r),(1,r)\}\subseteq (\mathbb R^2,\Vert\cdot\Vert_2)$ and consider $x:=(1,0)$ and $y:=(0,0)$. Notice that $m_{x,y}$ is a $\Delta$-point because there exists an isometry $\alpha:[0,1]\longrightarrow M$ connecting $x$ and $y$ (namely $\alpha(t):=(t,0)$ for every $t\in [0,1]$). However, it is not a Daugavet point. To this end, pick $u:=(1,r)$ and $v:=(0,r)$, and notice that the set $A:=\{z\in M: d(u,z)+d(v,z)=d(u,v)\}=\{u,v\}$. In fact, given $z\in M\setminus \{u,v\}$ we get that $z=(t,0)$ for certain $t\in [0,1]$. Hence
$$d(u,z)+d(v,z)=\sqrt{t^2+r^2}+\sqrt{(1-t)^2+r^2}>t+(1-t)=1=d(u,v),$$
which proves that $A=\{u,v\}$. Moreover
$$d(x,u)+d(y,v)=2r<2=d(x,y)+d(u,v),$$
so by (3) in Theorem \ref{theo:charadaugapoint} we get that $m_{x,y}$ is not a Daugavet point.
\end{example}

\begin{remark}\label{rema:deltanodaugaeasy}
The results of \cite[Section 3]{ahlp} show that, in many classical Banach spaces, the concept of $\Delta$- and Daugavet point coincide. The first example of a $\Delta$-point which is not a Daugavet point \cite[Example 4.7]{ahlp} required a study of absolute normalised norms (which was pushed quite further in \cite{hpv}). See also \cite{almt} for more technical examples of Banach spaces containing $\Delta$-points which are not Daugavet points.

Theorem \ref{theo:charadaugapoint} together with Proposition \ref{prop:connectdelta} provides easy procedures to obtain metric spaces whose Lipschitz-free space has $\Delta$-points which are not Daugavet points.
\end{remark}

Our aim is now to get necessary conditions for a molecule $m_{x,y}$ to be a $\Delta$-point. We begin with the following preliminary lemma. 

\begin{lemma}\label{lemma:lemanecdeltaclose}
Let $x\neq y\in M$ be two points and $f \in S_{\Lip (M)}$ be given. If $m_{x,y}$ is a $\Delta$-point, then given a slice $S(f, \eps)$ of $B_{\mathcal{F} (M)}$ with $m_{x,y} \in S(f, \eps)$, there exist $u\neq v\in M$ with $f(m_{u,v})>1-\frac{\varepsilon}{2}$ and such that $d(u,v)<\frac{2\varepsilon}{(1-\varepsilon)^2}d(x,y)$.
\end{lemma}

\begin{proof}
Let us follow the proof of \cite[Lemma 3.7]{gpr2018}. To this end, define $g:=\frac{f+f_{xy}}{2}$, where $f_{x,y}$ is defined by the equation
\[f_{xy}(t):= \frac{d(x,y)}{2}\frac{d(t,y)-d(t,x)}{d(t,y)+d(t,x)}.\]

Since $f_{x,y}(m_{x,y})=1$, then $g$ satisfies that $\Vert g\Vert\geq g(m_{x,y})>1-\frac{\eps}{2}$. Since $m_{x,y}$ is a $\Delta$-point, by Lemma \ref{lem:delta_points_nested_slices}, there exists a slice $S(h, \eta)$ of $B_{\mathcal{F}(M)}$ such that $S(h, \eta) \subset S(g, 1-\frac{\eps}{2})$ and $\| m_{x,y} -z \| \geq 2 - \frac{\eps}{2}$ for every $z \in S(h, \eta)$. Pick $u \neq v \in M$ such that $m_{u,v} \in S(h,\eta)$. Then, in particular, 
$$\Vert m_{x,y}-m_{u,v}\Vert>2-\varepsilon.$$
On the one hand, note that $f(m_{u,v})>1- \frac{\varepsilon}{2}$ and $f_{xy}(m_{u,v})>1- \frac{\varepsilon}{2}$. In particular, $f_{xy}(m_{u,v})>1- \varepsilon$, so by \cite[Lemma 3.6]{gpr2018} we get that
$$(1-\varepsilon)\max\{d(x,u)+d(y,u), d(x,v)+d(y,v)\}<d(x,y).$$
To obtain the desired conclusion, we will prove that 
$$(1-\varepsilon)(d(x,y)+d(u,v))\leq \min\{d(x,u)+d(y,v), d(x,v)+d(y,u)\}.$$
Notice that $\Vert m_{x,y}+m_{v,u}\Vert>2-\varepsilon$ implies that there exists $g\in S_{\Lip(M)}$ such that $g(x)-g(y)>(1-\varepsilon)d(x,y)$ and $g(v)-g(u)>(1-\varepsilon)d(u,v)$. Hence
\[
\begin{split}
1\geq \frac{g(x)-g(u)}{d(x,u)}& =\frac{g(x)-g(y)+g(v)-g(u)+g(y)-g(v)}{d(x,u)}\\
& \geq \frac{(1-\varepsilon)(d(x,y)+d(u,v))-d(y,v)}{d(x,u)},
\end{split}\]
from where $(1-\varepsilon)(d(x,y)+d(u,v))\leq d(x,u)+d(y,v)$. Using the same argument taking into account that $f(x)-f(y)>(1-\varepsilon)d(x,y)$ and $f(u)-f(v)>(1-\varepsilon)d(u,v)$ we get that $(1-\varepsilon)(d(x,y)+d(u,v))\leq d(x,v)+d(y,u)$, as desired.

With the previous inequalities in mind, the conclusion of the lemma follows with the estimates from \cite[Lemma 3.7]{gpr2018}.\end{proof}

From the previous lemma and \cite[Lemma 1.4]{ik2004}, we obtain the following characterization of the molecules which are $\Delta$-points in a Lipschitz-free space.

\begin{theorem}\label{prop:deltapointlocal}
Let $x\neq y\in M$ be two points. Then $m_{x,y}$ is a $\Delta$-point if and only if for every slice $S=S(f,\alpha)$ containing $m_{x,y}$ with $\alpha<1$ and for every $\varepsilon>0$, there exist $u,v\in M$ with $0<d(u,v)<\varepsilon$ such that $m_{u,v}\in S$.
\end{theorem}

\begin{proof}
Suppose that $m_{x,y}$ is a $\Delta$-point. Pick $\varepsilon>0$ and let $0<\beta<\alpha$ such that $\frac{2\beta}{(1-\beta)^2}d(x,y)<\varepsilon$. By \cite[Lemma 1.4]{ik2004} we get the existence of $g\in S_{\Lip(M)}$ such that
$$m_{x,y}\in S(g,\beta)\subseteq S.$$
Since $g(m_{x,y})>1-\beta$, by Lemma \ref{lemma:lemanecdeltaclose} we get that there are $u\neq v\in M$ such that $g(m_{u,v})>1-\beta$ (and so $m_{u,v}\in S$) and such that
$$d(u,v)<\frac{2\beta}{(1-\beta)^2}d(x,y)<\varepsilon,$$
as desired. 

Conversely, let $S = S(f, \alpha)$ be a slice containing $m_{x,y}$ with $\alpha < 1$ and $\eps >0$ be given. Using the existence of such $u, v \in M$ repeatedly, we may find a sequence $(u_n, v_n)$ of points with $m_{u_n, v_n} \in S$ such that $0 < d(u_n, v_n) \rightarrow 0$. By Theorem \ref{theo:prelipuntoscer}, we can conclude that  $m_{x,y}$ is a $\Delta$-point. 
\end{proof}

Even though the previous is a complete characterization of when a given molecule $m_{x,y}$ in $\mathcal F(M)$ is a $\Delta$-point, we would want to obtain a condition which purely depend on the metric space $M$. In order to do so, let us obtain the following consequence of Proposition \ref{prop:deltapointlocal}.

\begin{corollary}\label{coro:interballs}
Let $M$ be a complete metric space and let $m_{x,y}$ a $\Delta$-point. Pick $0<r<d(x,y)$. Then, for every $\varepsilon>0$, we get that
$$B(x,r+\varepsilon)\cap B (y,d(x,y)-r+\varepsilon)\neq \emptyset.$$
In particular, if $M$ is compact, $S (x, r) \cap S (y, d(x,y) -r ) \neq \emptyset$. 
\end{corollary}

\begin{proof}
The proof follows the lines of (iii)$\Rightarrow$(i) of \cite[Lemma 3.4]{gpr2018}. Assume with no loss of generality that $d(x,y)=1$. Assume that there exists $0<r<1$ and $\varepsilon_0>0$ such that
$$B(x,r+\varepsilon_0)\cap B(y,1-r+\varepsilon_0)=\emptyset,$$
and let us prove that $m_{x,y}$ is not a $\Delta$-point. Notice that we can assume that $d(B(x,r+\varepsilon_0), B(y,1-r+\varepsilon_0))\geq \delta_0>0$. Now define $f_i:M\longrightarrow \mathbb R, i=1,2$ by the equations
\[
\begin{split}
f_1(t):=\max \left\{r-\frac{1}{1+\varepsilon_0}d(x,t),0\right\},\\
f_2(t):=\min \left\{-(1-r)+\frac{1}{1+\varepsilon_0}d(y,t),0\right\}.
\end{split}
\]
Notice that $\Vert f_i\Vert_L\leq \frac{1}{1+\varepsilon_0}$ since Lipschitz norm does not increase under taking maxima and minima of Lipschitz functions \cite[Proposition 1.32]{Weaver2}. Define $f:=f_1+f_2$, which is a Lipschitz function. Also $f(x)-f(y)=1=d(x,y)$, so $\Vert f\Vert\geq 1$. It is clear from the construction that $\{z\in M: f_1(z)\neq 0\}\subseteq B(x,r+\varepsilon_0)$ and $\{z\in M: f_2(z)\neq 0\}\subseteq B(y,1-r+\varepsilon_0)$. Define the slice
$$S:=\left\{\mu\in B_{\mathcal F(M)}: f(\mu)>\frac{1}{1+\varepsilon_0}\right\}.$$
Notice that if $m_{u,v}\in S$ then, up re-labeling the points $u$ and $v$, it follows that $u\in B(x,r+\varepsilon_0)$ and $v\in B(y,1-r+\varepsilon_0)$ (because otherwise $f(m_{u,v})\leq \max\{\Vert f_1\Vert_L,\Vert f_2\Vert_L\}\leq \frac{1}{1+\varepsilon_0}$). This implies that $d(u,v)\geq \delta_0$. By Proposition \ref{prop:deltapointlocal} we get that $m_{x,y}$ is not a $\Delta$-point, as desired.
\end{proof}

\begin{remark}\label{remark:abundeltapoint}
Let $M$ be a complete metric space. By combining Corollary \ref{coro:interballs} with \cite[Lemma 3.2]{gpr2018}, we have that if $m_{x,y}$ is a $\Delta$-point for every $x \neq y \in M$, then $M$ is a length space. Now, we have the following equivalent statements:
\begin{enumerate}
\item $M$ is a length space;
\item $m_{x,y}$ is a $\Delta$-point for every $x \neq y \in M$; 
\item $\mathcal{F} (M)$ has the Daugavet property. 
\end{enumerate} 
\end{remark} 
For the proof of $(1) \Leftrightarrow (3)$, see \cite[Theorem 3.5]{gpr2018}.

\begin{remark}\label{remark:convexdld2p}
According to \cite[Definition 5.1]{ahlp}, a Banach space $X$ is said to have the \textit{convex-DLD2P} if $B_X=\overline{\co}(\Delta)$, where $\Delta$ is the set of $\Delta$-points of $B_X$.

In general, the convex-DLD2P does not imply that every element of the unit sphere is a $\Delta$-point \cite[Corollary 5.6]{ahlp}. However, Remark \ref{remark:abundeltapoint} shows that, if $\Delta$ contains the set $\{m_{x,y}: x\neq y\in M\}$ (in particular, $\mathcal F(M)$ would trivially have the convex-DLD2P) then $\mathcal F(M)$ even enjoys the Daugavet property.
\end{remark}

As we have pointed out above, a (complete) metric space $M$ is length if, and only if, for every pair of points $x,y\in M$ and every $\varepsilon>0$ the following holds
$$B\left(x,\frac{d(x,y)}{2}+\varepsilon\right)\cap B\left(y,\frac{d(x,y)}{2}+\varepsilon\right)\neq \emptyset.$$
This motivates that a local version could be true; in other words, that the converse of Corollary \ref{coro:interballs} hold. However, the following example, due to Luis Garc\'ia-Lirola, shows that this is not the case.

\begin{example}\label{example:luisca}
Let $M:= \{ 0, 1, x_t : t \in [0,1] \} \subset (\mathbb{R}, d)$ with the metric $d(x_t, x_s) = \min \{ t + s , 2-t-s \}$, where $x_0 = 0$ and $x_1 = 1$. Then $M$ is complete. It is clear that $B(0,r ) \cap B(1, 1-r) = \{x_r\} \neq \emptyset$ for every $0< r<1$. However, $m_{0,1}$ is not a $\Delta$-point. 
To this end, assume to the contrary that $m_{0,1}$ is a $\Delta$-point. 
For $\alpha \in (0,\frac{1}{2})$, let us consider the map $f$ defined as $f(x_t) = 0$ for every $0 \leq t < \alpha$ and $f(x_t) = 1 -\alpha$ for every $1-\alpha < t \leq 1$. Observe that the the slope of $f$ is $1$. Now, extend $f$ by McShane to the Lipschitz map on $M$ and denote the extension by $\tilde{f}$. 
Notice that $m_{0,1} \in S(f, 2\alpha)$ and there exists a sequence $(u_n, v_n)$ of points in $M$ with $m_{u_n, v_n} \in S(f, 2\alpha)$ such that $0 < d(u_n, v_n) \rightarrow 0$. By the shape of the metric space $M$, we have that both $u_n$ and $v_n$ converge to $0$ or $1$. However, in either case, we have that $f(m_{u_n, v_n}) \rightarrow 0$, which is a contradiction. 
\end{example} 

In order to obtain a kind of converse of Proposition \ref{prop:connectdelta}, our strategy will be to work with complete metric spaces $M$ included in Banach spaces (so every pair of points $x,y\in M$ can be joined by geodesics in $X$) and then assume that the $diam(B(x,r+\varepsilon)\cap B(y,d(x,y)-r+\varepsilon))$ tends to $0$ when $\varepsilon$ tends to $0$, in order to guarantee that $M$ contains curves which are close to the geodesic which does exist in $X$. The first result in this line will require compactness on $M$ but a very natural condition on $X$. 

\begin{theorem}\label{theo:equideltacomconv}
Let $X$ be a strictly convex Banach space and $M$ be a compact subset of $X$. The following assertions are equivalent:
\begin{enumerate}
\item $m_{x,y}$ is a $\Delta$-point.
\item For every $0<r< \|x - y\| $ we get that
$$S(x,r)\cap S(y, \|x - y\|-r)\neq \emptyset.$$
\item $[x,y]\subseteq M$. In particular, the points $x$ and $y$ are connectable by an isometric curve.
\end{enumerate}
\end{theorem}

\begin{proof}
(3)$\Rightarrow$(1) follows from Proposition \ref{prop:connectdelta} and (1)$\Rightarrow$(2) follows form Proposition \ref{coro:interballs}. So it only remains to prove that (2) implies (3). To this end, pick $0<r<\|x - y\|$ and define $A:=S(x,r)\cap S(y,\|x - y\| -r)$.  
Let us prove that $S_X(x,r)\cap S_X(y,\|x - y\| -r)$ only contains one point, which is precisely $\big( \frac{r}{\|x - y\|} \big) x+\big(1-\frac{r}{\|x - y\|}\big)y$. Once this is proved, since $A\neq\emptyset$ and $A\subseteq S_X(x,r)\cap S_X (y,\|x - y\|-r)$, we obtain that $\big( \frac{r}{\|x - y\|} \big) x+ \big(1- \frac{r}{\|x - y\|}\big)y\in A$, so it is in $M$.

So, let us now go to prove that $S_X(x,r)\cap S_X(y,\|x - y\|-r)$ only contains one point. To this end, pick $u,v\in S_X(x,r)\cap S_X(y,\|x - y\|-r)$ and let us consider the midpoint  $z=\frac{u+v}{2}$. Notice that $z\in B_X (x,r)\cap B_X (y,\|x - y\|-r)$. Moreover, it follows that
$$\|x - z \| \geq \| x- y \| - \| y - z \| \geq \|x - y \| -(\|x - y \| -r)=r.$$
This implies that $z=\frac{u+v}{2}\in S_X(x,r)$. Since $X$ is strictly convex, we get that $u=v$. 
\end{proof}

Let us end the section by obtaining a version of Theorem \ref{theo:equideltacomconv} which, assuming a stronger condition on $X$, will allow us to remove compactness assumption on $M$. Let us consider the following definition.

\begin{definition}
A Banach space $X$ is said to be \emph{midpoint locally uniformly rotund (in short, MLUR)} if whenever $(x_n)$ and $(y_n)$ are sequences in $S_X$ and satisfy that $\frac{1}{2} (x_n + y_n)$ converges to some element in $S_X$, then $\|x_n - y_n \| \rightarrow 0$. 
\end{definition} 

It is not difficult to check that $X$ is MLUR if and only if whenever $(x_n)$ and $(y_n)$ are sequences in $X$ such that $\|x_n\|$ and $\| y_n \|$ tend to $1$ and $\frac{1}{2} (x_n + y_n)$ converges to some member of $S_X$, it follows that $\|x_n - y_n \| \rightarrow 0$. Recall that a locally uniformly convex Banach space is MLUR and a MLUR Banach space is strictly convex. For more theory about rotundity in Banach spaces, see \cite{megginson}.  

Let us start with the following preliminary lemma.

\begin{lemma}\label{lem:char_MLUR} 
Let $X$ be a Banach space. Then $X$ is MLUR if and only if whenever $(x_n)$ and $(y_n)$ are sequences in $X$ such that $\|x_n\|$ and $\| y_n \|$ tend to $1$ and $ r x_n +(1-r) y_n$ converges to some member of $S_X$ for some $0 < r < 1$, it follows that $\|x_n - y_n \| \rightarrow 0$.
\end{lemma} 

\begin{proof}
The ``if'' part is clear taking $r=\frac{1}{2}$, so let us prove the ``only if'' part. Take $(x_n), (y_n)$ and $r$ as in the statement of the lemma. Let $\alpha \neq \beta \in (0,1)$ be so that $\alpha + \beta = 2r$ and define $z_n := \alpha x_n + (1-\alpha) y_n$ and $w_n := \beta x_n + (1-\beta) y_n$. Then 
\[
\|z_n \| \leq \alpha \|x_n\| + (1-\alpha) \|y_n\| \rightarrow 1 \quad \text{and} \quad \| w_n\| \leq \beta \|x_n\| + (1-\beta) \|y_n\| \rightarrow 1.
\]
On the other hand,
\[
\| r x_n + (1-r) y_n \| \leq \frac{1}{2} ( \|z_n\| + \|w_n\| ) \leq r \|x_n\| + (1- r) \|y_n \| \rightarrow 1.
\]
Since the left side term converges to $1$ by assumptions we get that, passing to a subsequence if necessary, $\|z_n\|$ and $\|w_n\|$ tend to $1$. As $\frac{1}{2} (z_n + w_n) = r x_n + (1-r) y_n$ converges to some element of $S_X$, we conclude that $\|z_n - w_n \| \rightarrow 0$ as $n \rightarrow \infty$. So, $\|x_n -y_n \| = |(\alpha - \beta)^{-1}| \|z_n - w_n \| \rightarrow 0$. 
\end{proof} 

The previous lemma is the key to obtain the following result.

\begin{lemma}\label{lem:MLUR_diameter}
Let $X$ be a MLUR Banach space, and $x \neq y \in X$. Then for $0 < r < d(x,y)$, 
\[
\diam{\left(B_X \left(x, r+ \frac{1}{n}\right) \cap B_X \left(y, \|x-y\| -r+\frac{1}{n}\right)\right)} \rightarrow 0
\]
as $n \rightarrow \infty$. 
\end{lemma} 

\begin{proof}
By scaling and translating, we may assume that $x \in S_X$ and $y = -x$. Given $0 < r < 2$, let $x_n \in B_X \left(x, r+ \frac{1}{n}\right) \cap B_X \left(-x, 2-r+\frac{1}{n}\right)$ for each $n \in \N$. In other words, $\|x_n - x \| \leq r + \frac{1}{n}$ and $\| x_n + x \| \leq 2- r + \frac{1}{n}$ for every $n \in \N$. Then 
\[
2 \leq \| x_n - x \| + \|x_n + x \| \leq 2 + \frac{2}{n} \rightarrow 2;
\]
hence, passing to a subsequence if necessary, we have that $\|x_n - x \| \rightarrow r$ and $\|x_n + x \| \rightarrow 2-r$ as $n \rightarrow \infty$. Let us consider $z_n := (r+\frac{1}{n})^{-1} (x-x_n)$ and $w_n = (2-r+\frac{1}{n})^{-1} (x+x_n)$ in $B_X$. Then $\|z_n\| \rightarrow 1$, $\|w_n \| \rightarrow 1$ and $\frac{r}{2}z_n + \left( 1 -\frac{r}{2} \right) w_n \rightarrow x \in S_X$.
By Lemma \ref{lem:char_MLUR}, we conclude that $\|z_n - w_n \| \rightarrow 0$ as $n \rightarrow \infty$. This implies that $x_n \rightarrow (1-r) x$ as $n \rightarrow \infty$. 
\end{proof} 

Now we are ready to get the desired result.

\begin{theorem}\label{theo:MLUR_equideltacomconv}
Let $X$ be a MLUR Banach space and $M$ be a complete subset of $X$. For $x \neq y \in M$, the following assertions are equivalent:
\begin{enumerate}
\item $m_{x,y}$ is a $\Delta$-point.
\item For every $0<r<\|x-y\|$ we get that
$$B(x,r+\varepsilon)\cap B(y, \|x-y\|-r+\varepsilon)\neq \emptyset.$$
\item $[x,y]\subseteq M$. In particular, the points $x$ and $y$ are connectable by an isometric curve.
\end{enumerate}
\end{theorem}

\begin{proof}
(3)$\Rightarrow$(1) follows from Proposition \ref{prop:connectdelta} and (1)$\Rightarrow$(2) follows form Proposition \ref{coro:interballs}. So it only remains to prove that (2) implies (3). To this end, let $x_n \in B(x, r+ \frac{1}{n}) \cap B(y, \|x-y\|-r+\frac{1}{n})$ for each $n \in \N$. Note from Lemma \ref{lem:MLUR_diameter} that  
\[
\diam{\left(B\left(x, r+ \frac{1}{n}\right) \cap B\left(y, \|x-y\|-r+\frac{1}{n}\right)\right)} \rightarrow 0
\]
as $n \rightarrow \infty$. This implies that $x_n$ converges to some $x_0$ in $M$ (noting that $M$ is complete). Moreover, we have that $\|x_0 - x\| = r$ and $\| x_0 - y \| = \|x - y\| - r$. The strictly convexity of $X$ forces $x_0 = \big(1-\frac{r}{\|x-y\|} \big)x + \big(\frac{r}{\|x-y\|}\big) y$. This proves that the segment $[x,y]$ is contained in $M$. 
\end{proof}

\textbf{Acknowledgements:} 
The authors are deeply grateful to Luis Garc\'ia-Lirola for letting them to include his Example \ref{example:luisca} in the present paper and for further fruitful conversations on the topic of the paper. 
They also thank Sheldon Dantas for fruitful discussions and valuable comments and thank Vegard Lima for pointing out a mistake in a previous version of this paper.


\begin{thebibliography}{99}

\bibitem {ahntt} T. A.~Abrahamsen, P.~H\'ajek, O.~Nygaard, J.~Talponen and S.~Troyanski, \textit{Diameter 2 properties and convexity}, Stud. Math. \textbf{232}, 3 (2016), 227--242


\bibitem {ahlp} T.~A.~Abrahamsen, R.~Haller, V.~Lima, and K.~Pirk, \textit{Delta- and Daugavet-points in Banach spaces}, Proc. Edinbugh Math. Soc. \textbf{63} (2020), 475--496.

\bibitem {almt} T.~A.~Abrahamsen, V.~Lima, A.~Martiny and S.~Troyanski, \textit{Daugavet- and Delta-points in Banach spaces with unconditional bases}, preprint (2020). Available at ArXiV.org with reference: \href{https://arxiv.org/abs/2007.04946}{arXiv:2007.04946}.

\bibitem {ag2019} R. J. Aliaga and A. J. Guirao, \textit{On the preserved extremal structure of Lipschitz-free spaces}, Studia Math. \textbf{245} (2019), no. 1, 1--14.



\bibitem{appp2020}
	R. J. Aliaga, E. Perneck{\'{a}}, C. Petitjean and A. Proch{\'{a}}zka,
	\emph{Supports in Lipschitz-free spaces and applications to extremal structure}, J. Math. Anal. Appl. \textbf{489}, 1 (2020), article 124128.
	
\bibitem {ar2020} R. J. Aliaga and A. Rueda Zoca, \textit{Points of differentiability of the norm in Lipschitz-free spaces}, J. Math. Anal. Appl. \textbf{489} (2020), article 124171.


\bibitem {blrjca} J.~Becerra Guerrero, G.~L\'opez-P\'erez and A.~Rueda Zoca, \textit{Diametral diameter two properties in Banach spaces}, J. Convex Anal. \textbf{25}, 3 (2018), 817--840.

\bibitem {blr2018}  J.~Becerra Guerrero, G.~L\'opez-P\'erez and A.~Rueda Zoca, \textit{Octahedrality in Lipschitz-free Banach spaces}, Proc. R. Soc. Edinb. Sect. A Math. \textbf{148A}, 3 (2018), 447--460. 

\bibitem {cgmr2020} R. Chiclana, L. Garc\'ia-Lirola, M. Mart\'in and A. Rueda Zoca, \textit{Examples and applications of the density of strongly norm attaining Lipschitz maps} (2019), to appear in Rev. Mat. Iberoamericana. Available at ArXiV.org with reference: \href{https://arxiv.org/abs/1907.07698}{arXiv:1907.07698}.

\bibitem {cgmm} Y. S. Choi, D. Garc\'ia, M. Maestre and M. Mart\'in, \textit{The Daugavet equation for polynomials}, Stud. Math. \textbf{178}, 1 (2007), 63--82.

\bibitem {dau} I.K. Daugavet, \textit{On a property of completely continuous operators in the space C}, Uspekhi Mat.
Nauk \textbf{18} (1963), 157-158 (Russian).


\bibitem{gppr2018}
	L. Garc{\'{i}}a-Lirola, C. Petitjean, A. Proch{\'{a}}zka and A. Rueda Zoca,
	\emph{Extremal structure and duality of Lipschitz free spaces},
	Mediterr. J. Math. \textbf{15} (2018), 69.

\bibitem{gpr2018}
	L. Garc{\'{i}}a-Lirola, A. Proch{\'{a}}zka and A. Rueda Zoca,
	\emph{A characterisation of the Daugavet property in spaces of Lipschitz functions},
	J. Math. Anal. Appl. \textbf{464} (2018), no. 1, 473--492.
 
\bibitem{Godefroy_2015}
	G. Godefroy,
	\emph{A survey on Lipschitz-free Banach spaces},
	Comment. Math. \textbf{55} (2015), no. 2, 89--118.
	
	
\bibitem {hpv} R. Haller, K. Pirk, and T. Veeorg, \textit{Daugavet- and Delta-points in absolute sums of Banach spaces}, to appear in J. Conv. Anal. \textbf{27} (2020). Available at ArXiV.org with reference: \href{https://arxiv.org/abs/2001.06197}{ 	arXiv:2001.06197}.
	
\bibitem {ik2004} Y. Ivakhno and V. Kadets, \textit{Unconditional sums of spaces with bad projections}, Visn. Khark. Univ., Ser. Mat. Prykl. Mat. Mekh. 645(54) (2004) 30--35.
	
\bibitem {ikw} Y.~Ivakhno, V.~Kadets and D.~Werner, \textit{The Daugavet property for spaces of Lipschitz functions}, Math. Scand. \textbf{101}, 2 (2007), 261--279.
	
	
\bibitem {kkw} V. Kadets, N. J. Kalton, D. Werner, \textit{Remarks on rich subspaces of Banach spaces}, Stud. Math. \textbf{159}, 2 (2003), 195--206.	
	
\bibitem {kmmw} V. Kadets, M. Mart\'in, J. Mer\'i and D. Werner, \textit{Lipschitz slices and the Daugavet equation for Lipschitz operators}, Proc. Amer. Math. Soc. \textbf{143}, 2 (2015), 5281--5292.
	
\bibitem {kssw} V. Kadets, R. V. Shvidkoy, G. G. Sirotkin, and D. Werner. \textit{Banach spaces
with the Daugavet property}, Trans. Am. Math. Soc. \textbf{352}, 2 (2000), 855--873.
	
\bibitem {lr2020} J. Langemets and A. Rueda Zoca, \textit{Octahedral norms in duals and biduals of Lipschitz-free spaces}, J. Funct. Anal. \textbf{279} (2020), article 108557.
	
	\bibitem{megginson} R. E. Megginson, \textit{An Introduction to Banach Space Theory} (Springer, New York, 1998).

\bibitem {o2020} A. Ostrak, \textit{On the duality of the symmetric strong diameter 2 property in Lipschitz spaces}, preprint (2020). Available at ArXiV.org with reference: \href{https://arxiv.org/abs/2008.03163}{ 	 	arXiv:2008.03163}.

\bibitem{pr2018}
	A. Proch{\'{a}}zka and A. Rueda Zoca,
	\emph{A characterisation of octahedrality in Lipschitz-free spaces},
	Ann. Inst. Fourier (Grenoble) \textbf{68} (2018), no. 2, 569--588.

\bibitem {rtv} A.~Rueda Zoca, P.~Tradacete, and I.~Villanueva, \textit{Daugavet property in tensor product spaces}, J. Inst. Math. Jussieu (2019), 1-20. doi:10.1017/S147474801900063X.

\bibitem {sw} E. S\'anchez-P\'erez and D. Werner, \textit{Slice continuity for operators and the Daugavet property for bilinear maps}, Functiones et Approximatio \textbf{50}, 2 (2014), 251--269. 

\bibitem{Weaver2}
  N. Weaver,
	\emph{Lipschitz algebras}, 2nd ed.,
	World Scientific Publishing Co., River Edge, NJ, 2018.
	
\bibitem {werner} D. Werner, \textit{Recent progress on the Daugavet property}, Ir. Math. Soc. Bull. \textbf{46} (2001), 77--97.


\end{thebibliography}
\end{document}